\documentclass[pdflatex,sn-mathphys-num]{sn-jnl}

\usepackage{graphicx}%
\usepackage{multirow}%
\usepackage{amsmath,amssymb,amsfonts}%
\usepackage{amsthm}%
\usepackage{mathrsfs}%
\usepackage[title]{appendix}%
\usepackage{xcolor}%
\usepackage{textcomp}%
\usepackage{manyfoot}%
\usepackage{booktabs}%
\usepackage{algorithm}%
\usepackage{algorithmicx}%
\usepackage{algpseudocode}%
\usepackage{listings}%
\usepackage{enumitem}
\usepackage{chngcntr}
\usepackage{apptools}
\AtAppendix{\counterwithin{theorem}{section}}
\AtAppendix{\counterwithin{definition}{section}}

\theoremstyle{thmstyleone}%
\newtheorem{theorem}{Theorem}
\newtheorem{proposition}[theorem]{Proposition}%

\theoremstyle{thmstyletwo}%
\newtheorem{remark}{Remark}%

\theoremstyle{thmstylethree}%
\newtheorem{lemma}{Lemma}%

\newcommand{\rn}{\mathbb{R}^N}
\newcommand{\R}{\mathbb{R}}
\makeatletter
\def\namedlabel#1#2{\begingroup
    #2%
    \def\@currentlabel{#2}%
    \phantomsection\label{#1}\endgroup
}
\makeatother

\begin{document}

\title[Normalised solutions for $p$-Laplacian]{Normalised solutions for $p$-Laplacian equations with $L^p$-supercritical growth}


\author*[1]{\fnm{Raj Narayan} \sur{DHARA}}\email{raj.dhara@unipd.it,~raj.narayan.dhara@p.lodz.pl\\ORCiD:  0000-0001-6749-1500}
\author[2]{\fnm{Matteo} \sur{Rizzi}}\email{mrizzi1988@gmail.com,~matteo.rizzi@uniba.it\\ORCiD:0000-0001-6294-753X}


\affil*[1]{\orgdiv{Institute of Mathematics}, \orgname{Lodz University of Technology}, \orgaddress{\street{\.Z{}eromskiego 116},
 \city{\L\'od\'z}, \postcode{90-924},  \country{Poland}}}

\affil[2]{\orgdiv{Dipartimento di Matematica}, \orgname{Universit\`a degli studi di Bari Aldo Moro}, \orgaddress{\street{Via Edoardo Orabona, 4},
 \city{Bari}, \postcode{70125},  \country{Italy}}}


\abstract{For $N\ge 3$ and $2<p<N$, we find normalised solutions to the equation
\begin{align*}
-\Delta_p u+(1+V(x))|u|^{p-2}u+\lambda u&=|u|^{q-2}u\qquad\text{in $\rn$}\\
\|u\|_2&=\rho
\end{align*}
in the mass supercritical and Sobolev subcritical case, that is $q\in(p\frac{N+2}{N},\frac{Np}{N-p})$, at least if $\rho>0$ is small enough. The function $V\in L^{N/p}(\rn)$, which plays the role of potential, is assumed to be non-positive and vanishing at infinity. Moreover, we will prove the compactness of the embedding of the space of radial functions $W^{1,p}_{rad}(\rn)\subset L^q(\rn)$ for $p\in(1,N)$ and $q\in(p\frac{N+2}{N},\frac{Np}{N-p})$. 
}




\pacs[Acknowledgement]{The authors were partially supported by the grant DFG project (Justus Liebig university and 62202684). https://www.dfg.de/en/funded-projects.}

\maketitle

\section{Introduction }\label{intro}

For $N\ge 3$, $p\in(2,N)$, $q\in(p\frac{N+2}{N},p^*)$ and $r\in[1,p^*)$, with $p^*:=\frac{Np}{N-p}$, we are interested in the problem
\begin{equation}\label{eq:EL:r}
\begin{aligned}
&-\Delta_p u+(1+V(x))|u|^{p-2}u+\lambda
|u|^{r-2}u=|u|^{q-2}u\qquad\text{in}\,\rn ,\\
&\int_{\rn}|u|^r dx=\rho^r,\,u>0,
\end{aligned}
\end{equation}
where $\rho>0$ is a fixed parameter and $V:\rn\to\mathbb{R}$ is a fixed potential such that
\begin{equation}\label{limV=1}
\lim_{|x|\to\infty}V(x)=0.
\end{equation} From the physical point of view, it is particularly relevant to consider the case $r=2$, namely the problem
\begin{equation}\label{eq:EL:1}
\begin{aligned}
&-\Delta_p u+(1+V(x))|u|^{p-2}u+\lambda
u=|u|^{q-2}u\qquad\text{in}\, \rn ,\\
&\int_{\rn}u^2 dx=\rho^2,\,u>0.
\end{aligned}
\end{equation}
This case is particularly meaningful since \eqref{eq:EL:1} is related to the generalized nonlinear Schr\"odinger equation (NLSE)
\begin{align}\label{eq:nlse}
    -i\dfrac{\partial \Psi}{\partial t}=\Delta_p \Psi - (V(x)+1)|\Psi|^{p-2}\Psi + |\Psi|^{q-2}\Psi\quad \text{in}\ \mathbb{R}\times\rn ,
\end{align}
involving the $p$-Laplacian on a wave function $\Psi:\R\times\rn\to \mathbb{C}$. In fact, if we investigate the standing wave solutions of~\eqref{eq:nlse}, i.e., the solutions of the form
\begin{align}\label{eq:ansatz}
    \Psi(t,x)=e^{i\lambda t} u(x),
\end{align}
plugging the ansatz~\eqref{eq:ansatz} into~\eqref{eq:nlse} we obtain that $\Psi$ is a solution to \eqref{eq:nlse} if and only if $u$ is a solution to the first equation in~\eqref{eq:EL:1}.

For $p>2$, the problem~\eqref{eq:EL:1} arises in many fields of mathematical physics, for instance, filtration process of an ideal incompressible fluid through a porous medium, non-Newtonian fluids, pseudo-plastic fluids, nonlinear elasticity, quantum fields, plasma physics, nonlinear optics and reaction diffusion. Moreover, the presence of the external potential in Schr\"odinger equation $V(x)$ modifies the wave function $\Psi(t,x)$.

In view of the ansatz~\eqref{eq:ansatz}, one may prescribe either the frequency $\lambda$ or the $L^2(\rn)$-norm of $u$. Note that, due to the conservation of mass, i.e. the fact that the $\|\Psi(t,\cdot)\|_{L^2(\rn)}$-norm of a standing wave solution is independent
of $t\in\mathbb{R}$, it is particularly interesting to consider solutions $u$ with prescribed $L^2(\rn)$-norm, which are known as \textit{normalised solutions}.\\ 

Here we are interested in finding normalised solutions to \eqref{eq:EL:1}, which can be characterised as the critical points of the functional 
\begin{align}\label{eq:func:1}
	J_V(u):=\frac{1}{p}\|u\|^p +\frac{1}{p}\int_{\rn}V(x)|u|^p\, dx -\frac{1}{q}\|u\|_q^q,\qquad\forall\, u\in W^{1,p}(\rn),
\end{align}
constrained to the sphere
$$\mathcal{S}_{\rho}:=\{u\in W^{1,p}(\rn)\cap L^2(\rn):\ \|u\|_{2}=\rho\},$$
where we have set
\[
\|u\|:=\left(\int_{\rn}(|\nabla u|^p +|u|^p )\, dx \right)^{\frac{1}{p}}\qquad\forall\, u\in W^{1,p}(\rn),
\]
and $\|u\|_s:=\left(\int_{\rn}|u|^s dx\right)^{1/s}$ for any $u\in L^s(\rn)$ and $s\in[1,\infty)$. We note that the functional $J_V$ is well defined in $W^{1,p}(\rn)\cap L^2(\rn)$ if, for instance, $V\in L^{\frac{N}{p}}(\rn)$. In fact $W^{1,p}(\rn)\subset L^q(\rn)$ for $q\in[p\frac{N+2}{N},p^*]$ and, due to the Sobolev embedding $W^{1,p}(\rn)\subset L^{p^*}(\rn)$, we have
\begin{equation}
\int_{\rn}|V(x)||u|^p dx\le \|V\|_{\frac{N}{p}}\|u\|^p_{p^*}\qquad\forall\,u\in W^{1,p}(\rn).
\end{equation} 
We consider the $L^p$-supercrtical and Sobolev-subcritical case, that is $q\in(p\frac{N+2}{N},p^*)$. This restriction is meaningful because, in the $L^p$-subcritical case $q\in(p,\frac{N+2}{N}p)$, the functional $J_V$ is coercive on the sphere $\mathcal{S}_{\rho}$, at least if $\|V\|_{\frac{N}{p}}$
is small enough, due to the Gagliardo-Nirenberg inequality
\begin{align}\label{eq:GN:1}
\|u\|_q\le C\|\nabla u\|_{p}^{\theta}\|u\|_2^{1-\theta},\quad \theta = \frac{Np(q-2)}{q(p(N+2)-2N)},\qquad\forall\, u\in W^{1,p}(\rn).
\end{align}
On the other hand, for $q\in(p\frac{N+2}{N},p^*)$, the functional $J_V$ is unbounded from below on $\mathcal{S}_{\rho}$, as a consequence there is no chance to find a solution to \eqref{eq:EL:1} by minimising $J_V$ on $\mathcal{S}_\rho$. We will see that this problem can be overcome by applying a mountain-pass strategy.\\

We will see that Problem (\ref{eq:EL:1}) is strictly related to the limit problem
\begin{equation}\label{limit-prob}
\begin{aligned}
&-\Delta_p u+|u|^{p-2}u+\lambda
u=|u|^{q-2}u\qquad\text{in}\,\rn\\
&\int_{\rn}u^2 dx=\rho^2,\,u>0,
\end{aligned}    
\end{equation}
which is treated in \cite{WLZL}, where the authors prove the existence of solutions $(\lambda,u)\in(0,\infty)\times(W^{1,p}(\rn)\cap L^2(\rn))$ for $\rho>0$ small enough. Their proof is based on a mountain-pass argument as well. More precisely, setting, for any $u\in W^{1,p}(\rn)\cap L^2(\rn)$,
$$I(u):=\frac{1}{p}\|u\|^p-\frac{1}{q}\|u\|_q^q,$$
they prove the following result.
\begin{theorem}[Theorem 1.1~\cite{WLZL}]\label{thWLZL}
Let $N\ge 3$, $2<p<N$ and $p\frac{N+2}{N}<q<p^*$. Then there exists $\rho_0>0$ such that, for any $\rho\in(0,\rho_0)$, there exists a constant $K_\rho>0$ and a solution $(\lambda_\rho,u_\rho)\in(0,\infty)\times ((W^{1,p}(\rn)\cap L^2(\rn))$ to problem (\ref{limit-prob}) such that $I(u_\rho)=c_\rho$, where
\begin{equation}\notag
\begin{aligned}
c_\rho&:=\inf_{\gamma\in\mathcal{G}_\rho}\max_{t\in[0,1]}I(g(t))>0,\\
\mathcal{G}_\rho&:=\{g\in C([0,1],\mathcal{S}_\rho):\,\|g(0)\|_{W^{1,p}(\rn)}\le K_\rho,\,I(g(1))<0\}.
\end{aligned} 
\end{equation}
\end{theorem}
\begin{remark}
\label{rem-least-energy-sol}
\begin{enumerate}
\item It is known that the solution $(\lambda_\rho,u_\rho)$ constructed in Theorem \ref{thWLZL} is the least-energy solution to the limit problem \eqref{limit-prob}, in the sense that any solution $(\lambda,v)\in \mathbb{R}\times (W^{1,p}(\rn)\cap L^2(\rn))$ to \eqref{limit-prob} fulfills $I(v)\ge I(u_\rho)$. In fact, if $(\lambda,v)$ solves the limit problem \eqref{limit-prob}, then $v$ belongs to the Pohozaev manifold
\begin{equation}
\mathcal{P}_\rho:=\{u\in \mathcal{S}_\rho:\, P(u)=0\},    
\end{equation}
where
\begin{equation}
 P(u):=\frac{p(N+2)-2N}{2p}\|\nabla u\|^p_p+\frac{N(p-2)}{2p}\|u\|_p^p-\frac{N(q-2)}{2q}\|u\|_q^q.  
\end{equation}
In addition, due to Lemma $2.3$ of \cite{WLZL}, we have 
$$c_\rho=\inf_{u\in\mathcal{P}_\rho} I(u)=I(u_\rho).$$
This gives $I(v)\ge I(u_\rho)$.
\item Moreover, by Lemma $2.8$ of \cite{WLZL}, the function $\rho\mapsto c_\rho$ is non-increasing in $(0,\infty)$.
\end{enumerate}
\end{remark}
The solution constructed in Theorem \ref{thWLZL} will be used in the proof of our main Theorem~\ref{th-Vle0}, in which we prove the existence of a mountain-pass type solution to problem (\ref{eq:EL:1}) under the assumptions
\begin{equation}\label{est-Vle0}
\begin{aligned}
&V\in L^{\frac{N}{p}}(\rn),\,\tilde{W}:=V(\cdotp)|\cdotp|\in L^{\frac{N}{p-1}}(\rn),\,V\le 0,\,\lim_{|x|\to\infty}V(x)=0,\,
\|V\|_{\frac{N}{p}}<S_p,\\
&
p^2S_p^{-\frac{p-1}{p}}\|\tilde{W}\|_{\frac{N}{p-1}}+N\max\left\{\frac{p}{2},(q-p-2)^+\right\}S_p^{-1}\|V\|_{N/p}< \min\left\{2\left(\frac{N}{q}-\frac{N-p}{p}\right),Nq-p(N+2)\right\},   
\end{aligned}
\end{equation}
where 
$$S_p:=\inf_{u\in D^{1,p}(\rn)}\frac{\|\nabla u\|^p_p}{\|u\|^p_{p^*}}=\pi^{-1/2}N^{-1/p}\left(\frac{p-1}{N-p}\right)^{1-\frac{1}{p}}\left(\frac{\Gamma(1+N/2)\Gamma(N)}{\Gamma(N/p)\Gamma(1+N-N/p)}\right)^{1/N}>0$$ 
is the Aubin-Talenti constant, or equivalently $S_p^{-\frac{1}{p}}$ is the best constant in the Sobolev embedding $D^{1,p}(\rn)\subset L^{p^*}(\rn)$, being $$D^{1,p}(\rn):=\{u\in L^{p^*}(\rn):\,\nabla u\in L^p(\rn)\}.$$
\begin{theorem}\label{th-Vle0}
Let $N\ge 3$, $2<p<N$ and $p\frac{N+2}{N}<q<p^*$. Assume that $V\ne 0$ fulfills (\ref{est-Vle0}). Then there exists $\rho_0>0$ such that, for any $\rho\in(0,\rho_0)$, there exists a solution $(\lambda,u)\in(0,\infty)\times ((W^{1,p}(\rn)\cap L^2(\rn))$ to problem (\ref{eq:EL:1}) such that $0< J_V(u)<c_\rho$. 
\end{theorem}
Similar results were proved in \cite{cingolani2019stationary,PR} for the Schr\"{o}dinger-Poisson equation in dimension $2$ and $3$ respectively, that is a semilinear non-local equation involving the classical Laplacian, namely the case $p=2$, and in \cite{bartsch2021normalized} in the case $p=2$ without any non-local term.\\

\begin{remark}\label{rem-rad}
It would be interesting to understand whether a radial solution 
in the case $V=0$ can be constructed and, if it is the case, whether it coincides with the one constructed in \cite{WLZL}. 
Moreover, it would be interesting to understand whether any solution $u\in W^{1,p}(\rn)\cap L^2(\rn)$ to \eqref{lim-eq} is radial and whether the radial solution is unique or not. This is known for $p=2$, while these problems are open for $p\ne 2$.
\end{remark}
 The open problem stated in Remark \ref{rem-rad}  is very relevant since the symmetry properties are often very helpful in proving the existence of solutions to PDEs. This can be seen in many papers available in the literature, such as \cite{R2, malchiodi2018periodic, agudelo2015solutions, agudelo2022doubling, agudelo2022k}, where some solutions to semilinear PDEs are constructed by prescribing the zero-level set and exploiting its symmetry properties.\\

In order to solve these open problems, we think it is useful to apply the following compactness Lemma for radial functions, which is a generalisation of the Strauss Theorem, that is a compactness result proved in \cite{kavian1993introduction} for the embedding $H^{1}_{rad}(\rn)\subset L^{q}(\rn)$ with $q\in(2,2^*)$ (see Theorem $1.2$ there).
\begin{lemma}\label{lemma-compact-rad}
Let $p\in(1,N)$ and $q\in(p,p^*)$. Then the embedding $W^{1,p}_{rad}(\rn)\subset L^q(\rn)$ is compact.    
\end{lemma}
We stress that Lemma \ref{lemma-compact-rad} holds true for any $p\in(1,N)$.\\


In the case $r\ne 2$ in \eqref{eq:EL:r}, some results for $V=0$ are proved in \cite{serrin2000uniqueness}, where existence and uniqueness of a radial ground state solution is proved for $p\in (1,N)$ and $1<p\le r<q<p^*$. The existence of normalised solutions with non-trivial potentials remains open. However, we give a decay estimate for radial positive solutions to \eqref{eq:EL:r} in the case $V=0$, $p\in(1,N)$ and $1<p<q<r<p^*$, that is we prove that
\begin{equation}\label{est-rad-r-constraint}
|u(x)|\le c|x|^{-\frac{p}{q-p}}  \qquad\forall\,x\in\rn. 
\end{equation}
The strategy of the proof is similar to the one adapted in Proposition $2.2$ of \cite{rizzi2024some} to obtain a decay estimate for radial positive solutions to some semilinear elliptic equations whose nonlinearity behaves like a power in a right neighbourhood of the origin.\\

The plan of the paper is the following. The main results are proved in Section \ref{sec-Vle0}. In particular, in Subsection \ref{subsec-compactness} we prove a very useful Lemma which describes the asymptotic behaviour of the bounded Palais-Smale sequences of the functional $J_V(\cdotp)+\frac{\lambda}{2}\|\cdotp\|_2^2$ with $\lambda>0$, known as the \textit{splitting Lemma}. Then in Subsection \ref{subsec-bdPS} we conclude the proof of Theorem \ref{th-Vle0}, with the aid of the splitting Lemma. Then in Subsection \ref{subsec-rad} we prove Lemma \ref{lemma-compact-rad} for radial functions. The decay estimate \eqref{est-rad-r-constraint} is proved in the Appendix.

\section{The proofs of the main results}\label{sec-Vle0}
\subsection{The splitting Lemma \ref{splitting-lemma}}\label{subsec-compactness}

In order to prove Theorem \ref{th-Vle0}, we need to describe the behaviour as $n\to\infty$ of the bounded Palais-Smale sequences of the functional $$J_{V,\lambda}(u):=J_V(u)+\lambda\int_{\R^N}u^2 dx,\qquad\forall\,u\in W^{1,p}(\rn)\cap L^2(\rn).$$
This can be done by means of the \textit{splitting Lemma}.\\

This Lemma will be proved in a slightly more general context, that is we take $p\in[2,N)$ and a potential $V$ fulfilling the following assumptions\\
\begin{enumerate}[label={($V_\arabic*$)}]
\item $V=V^+-V^-$, with $V^\pm\in L^{q^\pm}(\rn)$ for some $q^\pm\in\left[\frac{N}{p},\infty\right]$, $V^\pm\ge 0$, $V^+ V^-=0$.\label{V1}
\item $\|V^-\|_\infty<\lambda$ if $q^-=\infty$ or
$$\|V^-\|_{q^-}<S(p,q^-),\qquad\text{where}\,S(p,q):=\inf\frac{\|\nabla u\|^p_p}{\|u\|^p_{\frac{pq}{q-1}}},\,\forall\,q\in[p,p^*].$$
We note that $S(p,q)^{-1/p}$ is the best constant in the Sobolev embedding $D^{1,p}(\rn)\subset L^q(\rn)$.\label{V2}
\item $\lim_{|x|\to\infty}V^\pm(x)=0$.\label{V3}
\end{enumerate}



We note that, under these assumptions about $V$, the splitting Lemma \ref{splitting-lemma} that we prove here is a generalisation of the splitting Lemma proved in [Peng-Rizzi] and applies to more general situations than the ones needed here.
We define
\[
I_{\lambda}(u):= I(u)+\frac{\lambda}{2}\int_{\rn}u^2\, dx\qquad \forall\, u\in W^{1,p}(\rn)\cap L^2(\rn).
\]

\begin{lemma}[Splitting Lemma]
\label{splitting-lemma}
Assume that $V$ satisfies $(V_1),\,(V_2)$ and $(V_3)$. Let $\{u_n\}\subset W^{1,p}(\rn)\cap L^2(\rn)$ be a bounded $(PS)$ sequence for $J_{V,\lambda}$ such that $u_n\rightharpoonup u$ in $W^{1,p}(\rn)\cap L^2(\rn)$. Then either $u_k\to u$ strongly in $W^{1,p}(\rn)\cap L^2(\rn)$ or there exists an integer $k\geq 1$, $k$ non-trivial solutions $w^1,\cdots,w^k\in W^{1,p}(\mathbb{R}^N)\cap L^2(\rn)$ to the limit equation 
\begin{equation}\label{lim-eq}
-\Delta_p u+\lambda u+|u|^{p-2}u=|u|^{q-2}u\quad\text{in}\quad\rn
\end{equation}
and $k$ sequences $\{y_n^j\}_n\subset \rn,1\leq j\leq k$ such that $|y_n^j|\rightarrow\infty$ and $|y^j_n-y^i_n|\to\infty$ for $i\ne j$ as $n\rightarrow\infty$ and a subsequence, still denoted by $u_n$, such that
\begin{equation}\label{eq2.1}u_n=u+\sum_{j=1}^kw^j(\cdot-y_n^j )+o(1)~\text{strongly~in}~W^{1,p}(\rn)\cap L^2(\rn).
\end{equation}
Furthermore, one has
\begin{equation}\label{2.2}
\|u_n\|_2^2=\|u\|_2^2+\sum_{j=1}^k\|w^j\|_2^2+o(1)
\end{equation}
and
\begin{equation}\label{2.3}
J_{V,\lambda}(u_n)= J_{V,\lambda}(u)+\sum_{j=1}^k I_{\lambda}(w^j)+o(1) .
\end{equation}   
\end{lemma}
\begin{proof}
The proof is similar to the one of the splitting Lemma $2.6$ in \cite{PR}. However, since some meaningful differences appear, it is worth writing it down in a detailed form.\\

First we note that, since $\psi^1_n:=u_n-u$ is bounded in $W^{1,p}(\rn)\cap L^2(\rn)$ and $\psi^1_n\rightharpoonup 0$ in $W^{1,p}(\rn)\cap L^2(\rn)$, then $\psi^1_n$ is bounded in $L^{\bar{p}}(\rn)$ for any $\bar{p}\in[p,p^*]$, $p\ge 2$,
and $\psi^1_n\rightarrow 0$ in $L^{\bar{p}}_{\rm loc}(\rn)$.
In particular, $\psi^1_n\to 0$ pointwise a.e. in $\rn$.

Moreover, it is also possible to see that $\nabla \psi^1_n\to 0$ pointwise a.e. in $\R^N$. In order to prove this fact, we take a radial cutoff function $\varphi\in C^\infty(\rn)$ such that $\varphi=1$ in $B_1(0)$ and $\varphi=0$ in $\rn\backslash B_2(0)$, we set, for $R>0$, $\varphi_R(x):=\varphi(x/R)$. 
We note that
$$\nabla J_V(u_n)[\varphi_R(u_n-u)]=o_n(1),$$
since $u_n$ is bounded in $W^{1,p}(\rn)\cap L^2(\rn)$ and $\nabla J_V(u_n)\to 0$ in $(W^{1,p}(\rn)\cap L^2(\rn))'$. Moreover, using that $V^\pm\in L^{q^\pm}(\rn)$ for some $q^\pm\in[\frac{N}{p},\infty]$ and $u_n\to u$ strongly in $L^{\bar{p}}_{\rm loc}(\rn)$, we can see that
\begin{equation}\notag
\begin{aligned}
\nabla J_V(u_n)[\varphi_R(u_n-u)]&=\int_{B_{2R}(0)}|\nabla u_n|^{p-2}\nabla u_n\cdotp\nabla(u_n-u)\varphi_R dx\\
&+\int_{B_{2R}(0)\setminus B_R(0)}|\nabla u_n|^{p-2}(u_n-u)\nabla u_n\cdotp\nabla\varphi_R dx\\
&+\int_{B_{2R}(0)}|u_n|^{p-2}u_n(u_n-u)\varphi_R (1+V(x))dx\\
&-\int_{B_{2R}(0)}|u_n|^{q-2}u_n(u_n-u)\varphi_R dx\\
&=\int_{B_{2R}(0)}|\nabla u_n|^{p-2}\nabla u_n\cdotp\nabla(u_n-u)\varphi_R dx+o_n(1).
\end{aligned}   
\end{equation}
Thereby $$\int_{B_{2R}(0)}|\nabla u_n|^{p-2}\nabla u_n\cdotp\nabla(u_n-u)\varphi_R dx=o_n(1),$$
so that
$$\int_{B_{2R}(0)}(|\nabla u_n|^{p-2}\nabla u_n-|\nabla u|^{p-2}\nabla u)\cdotp\nabla(u_n-u)\varphi_R dx=o_n(1),$$
since $\nabla u_n\rightharpoonup\nabla u$ weakly in $L^{p}(\rn)$. As a consequence, using the inequality
$$(|\xi|^{p-2}\xi-|\zeta|^{p-2}\zeta)\cdotp(\xi-\zeta)\ge C|\xi-\zeta|^p,\qquad\forall\,\xi,\zeta\in\rn,$$
for some constant $C>0$, we have
$$\int_{B_{2R}(0)}|\nabla(u_n-u)|^p \varphi_Rdx=o_n(1),$$
which shows that $\nabla u_n\to\nabla u$ in $L^p_{\rm loc}(\rn)$, so in particular $\nabla u_n\to\nabla u$ pointwise a.e. in $\rn$.

As a consequence, the Brezis-Lieb Lemma yields that
\begin{equation}\label{BL}
\begin{aligned}
\|\nabla \psi^1_n\|^p_p&=\|\nabla u_n\|^p_p-\|\nabla u\|^p_p+o_n(1),\\
\|\psi^1_n\|^p_p&=\|u_n\|^p_p-\|u\|^p_p+o_n(1),\\
\|\psi^1_n\|^q_q&=\|u_n\|^q_q-\|u\|^q_q+o_n(1),\\
\|\psi^1_n\|^2_2&=\|u_n\|^2_2-\|u\|^2_2+o_n(1).
\end{aligned}    
\end{equation}
Using that $\psi^1_n\rightharpoonup 0$ in $W^{1,p}(\rn)\cap L^2(\rn)$, we have $\psi^1_n\to 0$ strongly in $L^s_{loc}(\rn)$ for any $s\in[p,p^*]$. As a consequence, due to the assumptions 
\ref{V1} and \ref{V3},
we have
\begin{equation}
\label{pot-to0}
\int_{\rn}V(x)|\psi^1_n|^pdx\to 0\qquad n\to\infty.    
\end{equation}
In fact, for any $\varepsilon>0$, there exists $R>0$ such that $|V(x)|<\varepsilon$ for $|x|>R$. As a consequence there exists $n_0(\varepsilon)>0$ such that, for any $n\ge n_0(\varepsilon)$ we have
\begin{equation}\notag
\begin{aligned}
\int_{\rn}V^+(x)|\psi^1_n|^p dx&=\int_{B_R(0)}V^+(x)|\psi^1_n|^p dx+ \int_{\rn\setminus B_R(0)}V^+(x)|\psi^1_n|^p dx \\
&<\|V^+\|_{q^+}\|\psi^1_n\|^2_{L^\frac{pq^+}{q^+-1}(B_R(0))}+\varepsilon\|\psi^1_n\|_p\le c\varepsilon.
\end{aligned}
\end{equation}
The term involving $V^-$ can be estimated similarly. As a consequence, using \eqref{pot-to0} along with the facts that $u_n$ is bounded in $W^{1,p}(\rn)\cap L^2(\rn)$ and $u_n\to u$ in $L^s_{loc}(\rn)$ for $s\in[p,p^*]$, the dominated convergence Theorem, \ref{V1} and \ref{V3}
yield that
\begin{equation}
\int_{\rn}V|\psi^1_n|^p dx=\int_{\rn}V|u_n|^p dx - \int_{\rn}V|u|^p dx +o_n(1).
\end{equation}
This fact, together with \eqref{BL}, yield that
\begin{equation}
\label{as-beh-J-psi}
I_\lambda(\psi^1_n)=J_{V,\lambda}(\psi^1_n)+o_n(1)=J_{V,\lambda}(u_n)-J_{V,\lambda}(u)+o_n(1).
\end{equation}

Moreover, we can see that
\begin{equation}
\label{as-beh-nabla-psi}
o_n(1)=\nabla J_{V,\lambda}(u_n)[\psi_n]-\nabla J_{V,\lambda}(u)[\psi_n]  =\nabla I_\lambda(\psi_n)[\psi_n]+o_n(1). 
\end{equation}
In order to prove \eqref{as-beh-nabla-psi}, the followings are sufficient to check:\\
\begin{align}
\int_{\rn}(|\nabla u_n|^{p-2}\nabla u_n-|\nabla u|^{p-2}\nabla u)\cdotp\nabla(u_n-u) dx=\int_{\rn}|\nabla\psi_n|^p dx+o_n(1),\label{grad:BL}\\
\int_{\rn}(1+V(x))(|u_n|^{p-2}u_n-|u|^{p-2}u)(u_n-u)dx=\int_{\rn}|\psi_n|^p dx+o_n(1),\label{V:BL}\\
\int_{\rn}(|u_n|^{q-2}u_n-|u|^{q-2}u)(u_n-u)dx=\int_{\rn}|\psi_n|^q dx+o_n(1)\label{psi-Lq}.
\end{align}
To check with~\eqref{grad:BL}, we have successively
\begin{align*}
   & \int_{\rn}(|\nabla u_n|^{p-2}\nabla u_n-|\nabla u|^{p-2}\nabla u)\cdotp\nabla(u_n-u)\, dx \\
    &= \int_{\rn}|\nabla u_n|^{p}\, - \int_{\rn}|\nabla u|^{p-2}\nabla u\cdot \nabla u_n \, dx - \int_{\rn}|\nabla u_n|^{p-2}\nabla u_n\cdot \nabla u \, dx +\int_{\rn}|\nabla u|^{p}\, dx\\
    &= \int_{\rn}|\nabla u_n|^{p} - \int_{\rn}|\nabla u|^{p} +o_n(1)
    - \int_{\rn}|\nabla u_n|^{p-2}\nabla u_n\cdot \nabla u \, dx +\int_{\rn}|\nabla u|^{p}\, dx\\
    &\stackrel{\eqref{BL}}{=}
    \int_{\rn}|\nabla\psi_n|^p - \int_{\rn}|\nabla u_n|^{p-2}\nabla u_n\cdot \nabla u \, dx +\int_{\rn}|\nabla u|^{p}\, dx+o_n(1),
\end{align*}
so then it remains to prove that as $n\to\infty$,
\[\int_{\rn}|\nabla u_n|^{p-2}\nabla u_n\cdot \nabla u \, dx \to \int_{\rn}|\nabla u|^{p}\, dx.\]
For this purpose we can see that for any $R>0$
\begin{equation}
\label{est-nabla-J}
\begin{aligned}
    \int_{\rn}(|\nabla u_n|^{p-2}\nabla u_n-|\nabla u|^{p-2}\nabla u)\cdot \nabla u \, dx
    &=\int_{B_R(0)}(|\nabla u_n|^{p-2}\nabla u_n-|\nabla u|^{p-2}\nabla u)\cdot \nabla u \, dx\\
    &+\int_{\R^N\setminus B_R(0)}(|\nabla u_n|^{p-2}\nabla u_n-|\nabla u|^{p-2}\nabla u)\cdot \nabla u \, dx .
\end{aligned}
\end{equation}
In order to show that the integral outside the ball is small, it is enough to observe that, for any $\varepsilon>0$, there exists $R_0(\varepsilon)>0$ such that
\begin{align}\notag
&\left|\int_{\R^N\setminus B_R(0)}(|\nabla u_n|^{p-2}\nabla u_n-|\nabla u|^{p-2}\nabla u)\cdot \nabla u  dx\right|\\
&\le\left(\int_{\R^N\setminus B_R(0)}|\nabla u|^p dx\right)^\frac{1}{p}\left(\int_{\R^N\setminus B_R(0)}|\nabla u_n|^p dx\right)^\frac{1}{p'}+\int_{\R^N\setminus B_R(0)}|\nabla u|^p dx<\varepsilon
\end{align}
for any $n$ provided $R\ge R_0(\varepsilon)$, since $\nabla u_n$ is bounded in $L^p(\R^N)$, which yields that $|\nabla u_n|^{p-2}\nabla u_n$ is bounded in $L^{p'}(\rn)$, where $p'=p/(p-1)$. Moreover, recalling the fact that $\nabla u_n\to\nabla u$ pointwise a.e. in $\rn$, we can see that for any $\varepsilon>0$ there exists $n_0(\varepsilon)>0$ such that 
$$\left|\int_{B_R(0)}(|\nabla u_n|^{p-2}\nabla u_n-|\nabla u|^{p-2}\nabla u)\cdotp\nabla u \, dx\right|<\varepsilon\qquad\forall\,n>n_0(\varepsilon).$$
\\

Relations \eqref{psi-Lq} and \eqref{V:BL} are proved similarly.\\

If $\psi^1_n\to 0$ strongly in $W^{1,p}(\rn)\cap L^2(\rn)$ there is nothing to prove, otherwise it is possible to show that there exists a sequence $\{y^1_n\}\subset\rn$ and a solution $w^1\ne 0$ to the limit equation \eqref{lim-eq} such that $|y^1_n|\to\infty$ and $\psi^1_n(\cdotp+y^1_n)\rightharpoonup w^1$ weakly in $X$.\\

In order to prove such a claim, first we note that there exists $\alpha>0$ such that
\begin{equation}
\label{low-est-I-psi}
I_\lambda(\psi_n^1)\ge \alpha\qquad\forall\,n\in\mathbb{N}. 
\end{equation}
In fact, if we assume by contradiction that, up to a subsequence, $I_\lambda(\psi^1_n)\to 0$, then \eqref{as-beh-nabla-psi} gives
$$I_\lambda(\psi^1_n)=\left(\frac{1}{p}-\frac{1}{q}\right)\|\psi^1_n\|^p+\lambda\left(\frac{1}{2}-\frac{1}{q}\right)\|\psi^1_n\|_2^2=o_n(1),$$
which yields that $\psi^1_n\to 0$ in $W^{1,p}(\rn)\cap L^2(\rn)$, a contradiction.\\

As a consequence, splitting $\rn$ into the countable union of unit cubes centered at integer points and setting $$d_n:=\max_{i}\|\psi^1_n\|_{L^q(Q_i)},$$
we have 
\begin{equation}
\label{dn-unif-pos}
d_n\ge \kappa \qquad\forall\,n\in\mathbb{N},
\end{equation}
for some suitable constant $\kappa>0$. In fact, using \eqref{low-est-I-psi} and \eqref{as-beh-nabla-psi}, we can see that
\begin{equation}\notag
\begin{aligned}
\frac{\alpha}{2}&\le I_\lambda(\psi_n)+o(1)=-\left(\frac{1}{2}-\frac{1}{p}\right)\|\psi^1_n\|^p+\left(\frac{1}{2}-\frac{1}{q}\right)\|\psi^1_n\|_q^q\\
&\le\left(\frac{1}{2}-\frac{1}{q}\right)\|\psi^1_n\|_q^q\le\left(\frac{1}{2}-\frac{1}{q}\right)d_n^{q-p}\sum_{i}\|\psi^1_n\|_{L^q(Q_i)}^p\\
&\le S(p,q)^{-1}\left(\frac{1}{2}-\frac{1}{q}\right)d_n^{q-p}\sum_{i}\|\psi^1_n\|^p_{W^{1,p}(Q_i)}\le\tilde{c}(p,q)d_n^{p-q},
\end{aligned}
\end{equation}
which proves \eqref{dn-unif-pos}.\\

Let $y^1_n$ be the centre of the cube $Q_i$ which achieves the maximum $d_n$. If $y^1_n$ were bounded, then it would be constant (up to a subsequence), which means that there exists $j$ such that
$$\|\psi^1_n\|^q_{W^{1,p}(Q_j)}\ge c\|\psi^1_n\|^q_{L^q(Q_j)}\ge \kappa>0,$$
which is impossible since $\psi^1_n\rightharpoonup0$ in $W^{1,p}(\rn)\cap L^2(\rn)$, which yields that $\psi^1_n\to 0$ in $L^{p}_{loc}(\rn)$. As a consequence $|y^1_n|\to\infty$ and $\psi^1_n(\cdotp+y^1_n)\rightharpoonup w^1$ in $W^{1,p}(\rn)\cap L^2(\rn)$, where $w^1\ne 0$ is a solution to the limit equation \eqref{lim-eq}. We stress that, for any $\rho\in(0,\rho_0)$, there exists $M=M_\rho>0$ such that
\begin{equation}\label{low-est-wj}
\|w\|  \ge M,
\end{equation}
for any solution $w\in W^{1,p}(\rn)\cap L^2(\rn)$ to \eqref{lim-eq} with $\|w\|_2\le \rho$. In fact, any solution $w\in W^{1,p}(\rn)\cap L^2(\rn)$ with $\|w\|_2\le \rho$ to \eqref{lim-eq} fulfills
\begin{equation*}
\begin{aligned}
&I(w)=\frac{1}{p}\|w\|^p-\frac{1}{q}\|w\|^q_q\ge I(u_{\|w\|_2})=c_{\|w\|_2}\ge c_\rho,\\
&\|w\|^p+\lambda\|w\|^2_2=\|w\|^q_q,
\end{aligned}
\end{equation*}
which give
$$\left(\frac{1}{p}-\frac{1}{q}\right)\|w\|^p\ge\left(\frac{1}{p}-\frac{1}{q}\right)\|w\|^p-\frac{\lambda}{q}\|w\|^q_q\ge c_\rho,$$
which gives \eqref{low-est-wj}.\\

Iterating this process, setting, for $j\ge 2$, $\psi^j_n:=\psi^{j-1}_n(\cdotp+y^{j-1}_n)-w^{j-1}$, it is possible to find a sequence $y^j_n$ such that $|y^j_n|\to\infty$, $|y^j_n-y^i_n|\to\infty$ for $1\le i<j$ and $\psi^j_n(\cdotp+y^j_n)\rightharpoonup w^j$ in $W^{1,p}(\rn)\cap L^2(\rn)$, where $w^j$ is a solution to \eqref{lim-eq} and
\begin{equation}\notag
\begin{aligned}
\|\nabla \psi^j_n\|^p_p&=\|\nabla \psi^{j-1}_n\|^p_p-\|\nabla w^{j-1}\|^p_p+o_n(1),\\
\|\psi^j_n\|^p_p&=\|\psi^{j-1}_n\|^p_p-\|w^j\|^p_p+o_n(1),\\
\|\psi^j_n\|^q_q&=\|\psi^{j-1}_n\|^q_q-\|w^j\|^q_q+o_n(1),\\
\|\psi^j_n\|^2_2&=\|\psi^{j-1}_n\|^2_2-\|w^j\|^2_2+o_n(1).
\end{aligned}    
\end{equation}
In particular, we have
\begin{equation}\notag
\begin{aligned}
\|\psi^j_n\|^2_2&=\|u_n\|^2_2-\|u\|^2_2-\sum_{i=1}^j\|w^i\|^2_2+o_n(1),\\
I_\lambda(\psi^j_n)&=J_{V,\lambda}(u_n)-J_{V,\lambda}(u)-\sum_{i=1}^{j-1}J_{V,\lambda}(u_n)+o_n(1) .
\end{aligned}
\end{equation}
Since $w^j$ satisfies \eqref{low-est-wj} for any $j\ge 1$, the iterative process must stop after a finite number of steps, that is either $k=0$ or there exists $k\ge 1$ such that $w^j\ne 0$ for $1\le j\le k$ and $w^j=0$ for $j>k$.
\end{proof}

\subsection{Proof of Theorem \ref{th-Vle0}.}\label{subsec-bdPS}
Let us denote, for $k>0$,
\begin{align}
\alpha_k:=\sup_{u\in D_k}J_V(u);\quad \beta_k:=\inf_{u\in \partial D_k}J_V(u),\\
D_k:= \{u\in \mathcal{S}_{\rho}:\  \|u\|\le k\},
\end{align}
where $\|\cdotp\|$ denotes the standard $W^{1,p}(\rn)$-norm.
\begin{lemma}
\label{lemma-MP}
Assume that $V\in L^{N/p}(\rn)$, and $\|V\|_{\frac{N}{p}}<S_p$. Then for $\frac{N+r}{N}p<q<p^*$, there exists $0<k_1<k_2$ such that
\begin{align}\label{eq:lev:1}
0<\alpha_{k_1}<\beta_{k_2}.
\end{align}
and $J_V(u)>0$ for any $u\in D_{k_1}\setminus\{0\}$.
\end{lemma}
\begin{proof}
By the H\"{o}lder inequality and the Sobolev embedding $D^{1,p}(\rn)\subset L^{p^*}(\rn)$
, we have
\begin{equation}\notag
\left|\int_{\rn}V(x)|u|^pdx\right|\le\|V\|_{\frac{N}{p}}\|u\|^p_{p^*}\le S_p^{-1}\|V\|_{\frac{N}{p}}\|\nabla u\|^p_p\le S_p^{-1}\|V\|_{\frac{N}{p}}\|u\|^p\qquad\forall\,u\in \mathcal{S}_{\rho}.
\end{equation}
As a consequence, since $V\le 0$ we have 
\begin{equation}\label{low-est-JV}
\begin{aligned}
J_V(u)&\ge \frac{1}{p}(1-S_p^{-1}\|V\|_{\frac{N}{p}})\|u\|^p
- \frac{1}{q}\|u\|_q^q\\
    &\stackrel{\eqref{eq:GN:1}}{\ge} \frac{1}{p}(1-S_p^{-1}\|V\|_{\frac{N}{p}})\|u\|^p
    - \frac{C^q}{q}\|\nabla u\|_p^{\theta q}\|u\|_2^{(1-\theta)q}\\
    &\ge \frac{1}{p}(1-S_p^{-1}\|V\|_{\frac{N}{p}})\|u\|^p-\frac{C^q}{q}\rho^{(1-\theta)q}\|u\|^{\theta q}=f(\|u\|),\qquad \forall\, u\in \mathcal{S}_{\rho},
\end{aligned}
\end{equation}
where we have set $$f(t):=\frac{1}{p}(1-S_p^{-1}\|V\|_{\frac{N}{p}})t^p-\frac{C^q}{q}\rho^{(1-\theta)q}t^{\theta q}.$$\\

We note that $f(0)=0$. Using that $1-S_p^{-1}\|V\|_{\frac{N}{p}}>0$, thanks to assumption (\ref{est-Vle0}), and $\theta q>p$, since $p\frac{N+2}{N}<q<p^*$, we can see that there exists a unique $t_0=t_0(p,q,N,\rho,\|V\|_{\frac{N}{p}})>0$ such that $$f(t_0)=\max_{t\in(0,\infty)}f(t)>0,$$ 
$f$ is strictly increasing in $(0,t_0)$ and strictly decreasing in $(t_0,\infty)$.\\

As a consequence, (\ref{low-est-JV}) yields that, setting $k_2=t_0$, we have $\beta_{k_2}\ge f(t_0)>0$ and $J_V(u)\ge f(\|u\|)>0$ for any $u\in D_{k_1}\setminus\{0\}$.

On the other hand, since $V\le 0$, taking $k_1<\min\{(pf(t_0))^{\frac{1}{p}},k_2\}$, we have
$$0\le J_V(u)\le \frac{1}{p}\|u\|^p\le\frac{k_1^p}{p}<f(t_0)\le\beta_{k_2}\qquad\forall\,u\in D_{k_1},$$
which concludes the proof.
\end{proof}
Lemma \ref{lemma-MP} can be used to prove that $J_V$ has the mountain pass geometry. More precisely, for $t>0$ and $u\in H^1(\rn)$, 
we define the scaling $u^t(x):=t^{\frac{N}{2}}u(tx)$ and we introduce the mountain-pass level
\begin{equation}
m_{V,\rho}:=\inf_{\gamma\in \Gamma_\rho}\max_{t\in[0,1]}J_V(\gamma(t)),
\end{equation}
where $\Gamma_\rho$ is the set of paths
$$\Gamma_\rho:=\{\gamma\in C([0,1],\mathcal{S}_\rho):\,\gamma(0)=u_\rho^{t_1},\,\gamma(1)=u_\rho^{t_2}\},$$
with $0<t_1<1<t_2<\infty$ such that $u_\rho^{t_1}\in D_{k_1}$ and $J_V(u_\rho^{t_2})<0$. We note that this is possible for $p>2$, since
$$\|u^t\|_p^p=t^{\frac{N(p-2)}{2}}\|u\|_p^p\to 0\qquad t\to 0^+$$
provided $p>2$.
\begin{lemma}\label{lemma-mpg}
Assume that $V\in L^{\frac{N}{p}}(\rn)$ and $\|V\|_{\frac{N}{p}}<S_p$. Then, in the above notations, we have
$$0<\sup_{\gamma\in\Gamma_\rho}\max\{J_V(\gamma(0)),\,J_V(\gamma(1))\}<m_{V,\rho}.$$
Moreover, if $V\le 0,\,V\ne 0$, we have $m_{V,\rho}<c_\rho$.
\end{lemma}
\begin{proof}
By definition of $\Gamma_\rho$, the definition of $\sup$, Lemma \ref{lemma-MP} and \eqref{low-est-JV}, we have $$\sup_{\gamma\in\Gamma_\rho}\max\{J_V(\gamma(0)),\,J_V(\gamma(1))\}=\max\{J_V(u_\rho^{t_1}),\,J_V(u_\rho^{t_2})\}=J_V(u_\rho^{t_1})>0\qquad\forall\,g\in \Gamma_\rho.$$ Now we fix $\gamma\in\Gamma_\rho$. Since $J_V(\gamma(1))<0$, Lemma \ref{lemma-MP} yields that $\|\gamma(1)\|>k_2$. Being $\|\gamma(0)\|\le k_1<k_2$, by continuity there exists $\bar{t}\in(0,1)$ such that $\|\gamma(\bar{t})\|=k_2$, hence $$\max_{t\in[0,1]} J_V(\gamma(t))\ge J_V(\gamma(\bar{t}))\ge\beta_{k_2}>\alpha_{k_1}\ge J_V(\gamma(0))=\max\{J_V(\gamma(0)),J_V(\gamma(1))\}.$$  
Since $\gamma$ is arbitrary, this yields that $$m_{V,\rho}\ge \beta_{k_2}>\alpha_{k_1}\ge\sup_{\gamma\in\Gamma_\rho}\max\{J_V(\gamma(0)),J_V(\gamma(1))\}.$$
In order to prove that $m_{V,\rho}<c_\rho$ if $V\ne 0$, we note that the path $$\tau\mapsto u_\rho^{(1-\tau)t_1+\tau t_2}$$
is in $\Gamma_\rho$. 
Moreover, we have $$c_\rho=I(u_\rho)=\max_{t>0}I(u_\rho^t)=\max_{t>0}\left(\frac{\|\nabla u_\rho\|_p^p}{p}t^{\frac{p(N+2)-2N}{2}}+\frac{\|u_\rho\|_p^p}{p}t^{\frac{N(p-2)}{2}}-\frac{\|u_\rho\|_q^q}{q}t^{\frac{N(q-2)}{2}}\right).$$
As a consequence
$$m_{V,\rho}\le\max_{\tau\in[0,1]} J_V(u_\rho^{(1-\tau)t_1+\tau t_2})\le \max_{t>0} J_V(u_\rho^t)<\max_{t>0}I(u^t_\rho)=c_\rho,$$
because $V\le 0,\, V\ne 0$.
\end{proof}
\begin{lemma}[Pohozaev identity]
\label{lemma-Poho}
Assume that $V\in L^{\frac{N}{p}}(\rn)$, $\tilde{W}\in L^{\frac{N}{p-1}}(\rn)$. Let $u\in \mathcal{S}_\rho$ be a solution to problem (\ref{eq:EL:1}). Then it satisfies the Pohozaev identity
\begin{equation}\label{Pohozaev-vle0}
\begin{aligned}
P_V(u):=&\frac{p(N+2)-2N}{2p}\|\nabla u\|^p_p+\frac{N(p-2)}{2p}\|u\|_p^p-\frac{N(q-2)}{2q}\|u\|_q^q \\
&+\int_{\rn}V(x)\left(\frac{Np}{4}|u|^p+\frac{p}{2}|u|^{p-2}u\nabla u\cdotp x\right) dx=0 .
\end{aligned}
\end{equation}
\end{lemma}
\begin{proof}
On the one hand we note that, for any $u\in W^{1,p}(\rn)\cap L^2(\rn)$, an explicit computation shows that
$$\frac{d}{dt}\bigg|_{t=1}J_V(u^t)=P(u).$$
On the other hand, if $u\in W^{1,p}(\rn)\cap L^2(\rn)$ is a solution to\eqref{eq:EL:1}, then
$$\frac{d}{dt}\bigg|_{t=1}J_V(u^t)=\nabla J_V(u)[\partial_t u^t|_{t=1}]=-\lambda\int_{\rn}u(\partial_t u^t|_{t=1})=-\lambda\frac{d}{dt}\left(\int_{\rn} u^2 dx\right)=0,$$
so that $P_V(u)=0$.
\end{proof}
We note that $P_V$ reduces to $P$ if $V=0$ (see Remark \ref{rem-least-energy-sol}).
\begin{lemma}
\label{lemma-en>0}
Assume that (\ref{est-Vle0}) is fulfilled. Let $u\in \mathcal{S}_{\rho}$ be a solution to 
\begin{equation}
\label{main-eq-not-normalised}
-\Delta_p u+(1+V(x))|u|^{p-2}u+\lambda u=|u|^{q-2}u\qquad\text{in}\,\rn
\end{equation}
Then $J_V(u)> 0$.
\end{lemma}
\begin{remark}
We stress that Lemma \ref{lemma-en>0} holds true for any solution $u$ to \eqref{main-eq-not-normalised}, without any requirement about $\|u\|_2$.
\end{remark}
\begin{proof}
Let $u\in \mathcal{S}_{\rho}$ be a solution to problem (\ref{limit-prob}). The Pohozaev identity (\ref{Pohozaev-vle0}) yields that
\begin{equation}
\begin{aligned}
J_V(u)=&\frac{1}{p}(\|\nabla u\|_p^p+\|u\|_p^p)+\frac{1}{p}\int_{\rn}V(x)|u|^p dx-\frac{1}{q}\|u\|_q^q\\
=&\frac{1}{p}(\|\nabla u\|_p^p+\|u\|_p^p)+\frac{1}{p}\int_{\rn}V(x)|u|^p dx-\\
&\frac{2}{N(q-2)}\left(\frac{p(N+2)-2N}{2p}\|\nabla u\|^p_p+\frac{N(p-2)}{2p}\|u\|_p^p+\int_{\rn}V(x)\left(\frac{N}{2}|u|^p+|u|^{p-2}u\nabla u\cdotp x\right)
dx\right)\\
=&\frac{Nq-p(N+2)}{Np(q-2)}\|\nabla u\|_p^p+\frac{q-p}{p(q-2)}\|u\|_p^p+\frac{q-2-p}{p(q-2)}\int_{\rn}V(x)|u|^p dx\\
&+\frac{2}{N(q-2)}\int_{\rn}V(x)|u|^{p-2}u\nabla u\cdotp x dx.
\end{aligned}
\end{equation}
By the H\"{o}lder inequality and the Sobolev embedding $D^{1,p}(\rn)\subset L^{p^*}(\rn)$ we have
\begin{equation}
\label{est-CD}
\begin{aligned}
&\left|\int_{\rn}V(x)|u|^p dx\right|\le \|V\|_{\frac{N}{p}}\|u\|_{p^*}^p\le S_p^{-1}\|V\|_{\frac{N}{p}}\|\nabla u\|_p^p,\\
&\left|\int_{\rn}V(x)|u|^{p-2}u\nabla u\cdotp x\right|\le\|\tilde{W}\|_{\frac{N}{p-1}}\|u\|_{p^*}^{p-1}\|\nabla u\|_p\le S_p^{-\frac{p-1}{p}}\|\tilde{W}\|_{\frac{N}{p-1}}\|\nabla u\|_p^p,  
\end{aligned}    
\end{equation}
which concludes the proof thanks to (\ref{est-Vle0}), which yields that
$$N(q-p-2)^+S_p^{-1}\|V\|_{\frac{N}{p}}+2pS_p^{-\frac{p-1}{p}}\|\tilde{W}\|_{\frac{N}{p-1}}< Nq-p(N+2).$$
\end{proof}
In order to construct a Palais-Smale sequence at level $m_{V,\rho}$, we will use the following general result, which was stated in \cite{PR}. For a proof, we refer to \cite{G1993}.
\begin{proposition}\label{pro4.1}
Let $M$ be a Hilbert manifold, $\mathcal{I}\in C^1(M,\mathbb{R})$ be a given functional and $K\subset M$ be compact. Suppose that the subset
\begin{equation*}
\mathcal{C}\subset \{C\subset M:C ~\text{is~compact},K\subset C\}
\end{equation*}
is homotopy--stable, i.e., it is invariant with respect to deformations leaving $K$ fixed. Moreover, let
\begin{equation*}
\max_{u\in K}\mathcal{I}(u)<c:=\inf_{C\subset  \mathcal{C}}\max_{u\in C} \mathcal{I}(u)\in \mathbb{R},
\end{equation*}
let $\{\sigma_n\}\subset\mathbb{R} $ be a sequence such that $\sigma_n\rightarrow0$ and $\{C_n\}\subset \mathcal{C}$ be a sequence such that
\begin{equation*}
0\leq \max_{u\in C_n}\mathcal{I}(u)-c\leq \sigma_n.
\end{equation*}
Then there exists a sequence $\{v_n\}\subset M$ satisfying
\begin{description}
\item[(1)] $|\mathcal{I}(v_n)-c|\leq \sigma_n$,
\item[(2)] $\|\nabla_M\mathcal{I}(v_n)\|\leq C_1\sqrt{\sigma_n}$,
\item[(3)] dist$(v_n,C_n)\leq C_2\sqrt{\sigma_n}$.
\end{description}
\end{proposition}
\begin{lemma}\label{lemma-PS-seq}
Assume that (\ref{est-Vle0}) holds. Then there exists a $(PS)$ sequence $\{u_n\}\subset \mathcal{S}_\rho$ such that as $n\rightarrow\infty,$
\begin{equation}\label{syst-PS-seq}
J_V(u_n)\rightarrow m_{V,\rho},\,\nabla _{\mathcal{S}_\rho}J_V(u_n)\rightarrow 0,\, P(u_n)\rightarrow 0.
\end{equation}
\end{lemma}
\begin{proof}
Let $\{\gamma_n\}_n\subset\Gamma_\rho$ be a sequence such that $$\max_{t\in[0,1]}J_V(\gamma_n(t))\le m_{V,\rho}+\frac{1}{n}.$$   
Using that $J_V(|u|)=J_V(u)$ for any $u\in W^{1,p}(\rn)\cap L^2(\rn)$, we can assume that $\gamma_n(t)\ge 0$ a. e. in $\rn$. Define
$$\tilde{J}_V(u,t):=J_V(u^t),\qquad\forall\,(u,t)\in (W^{1,p}(\rn)\cap L^2(\rn))\times(0,\infty)$$
and
\begin{equation*}
\tilde{\Gamma}_\rho=\{\tilde{\gamma}:[0,1]\rightarrow \mathcal{S}_\rho\times \R:\, \tilde{\gamma}(0)=(u_\rho^{t_1},1),~ \tilde{\gamma}(1)=(u_\rho^{t_2},1)\}.
\end{equation*}
Note that $\tilde{\gamma}_n(t):=(\gamma_n(t),1)\in \tilde{\Gamma}_\rho$, for any $\gamma\in\Gamma_\rho$. As a consequence, applying Proposition \ref{pro4.1} with $\mathcal{I}=\tilde{J}_V$ and
\begin{equation*}
M :=\mathcal{S}_\rho\times\R,~K:=\{(u_\rho^{t_1},1),(u_\rho^{t_2},1)\},~\mathcal{C}:=\tilde{\Gamma}_\rho,~C_n:=\{(\gamma_n(t),1):t\in[0,1]\},
\end{equation*}
there exists $(v_n,t_n)\in \mathcal{S}_\rho\times \mathbb{R}$ such that as $n\rightarrow\infty$,
\begin{equation*}
\tilde{J}_V(v_n,t_n)\rightarrow m_{V,a},~\nabla\tilde{J}_V(v_n,t_n)\rightarrow0.
\end{equation*}
Moreover, we get
\begin{equation*}
\min_{t\in[0,1]}(\|v_n-\gamma_n(t)\|+\|v_n-\gamma_n(t)\|_2
+|t_n-1|)\leq \frac{\tilde{C}}{\sqrt{n}},
\end{equation*}
thus $t_n\rightarrow 1$ and there exists $s_n\in [0,1]$ such that $$\|v_n-\gamma_n(s_n)\|+\|v_n-\gamma_n(s_n)\|_2
\rightarrow0\qquad n\rightarrow\infty.$$
\vskip0.1in
Define $u_n:=(v_n)^{t_n}.$ Since $\gamma_n(t)\geq0$ a.e. in $\mathbb{R}^N$, then $\|v_n^-\|_p\leq \|v_n-\gamma_n(s_n)\|_p=o(1)$, so that $v_n^-\rightarrow0$ a.e. in $\mathbb{R}^N$, up to a subsequence. So
\begin{equation}\label{eq2.15}
\|u_n^-\|_p\rightarrow0\quad\text{as}\quad n\rightarrow\infty.
\end{equation}
Now we show $\{u_n\}$ is a $(PS)$ sequence for $J_V$.\\

First we note that $J_V(u_n)\rightarrow m_{V,\rho}$  as $n\rightarrow\infty$.\\

Furthermore, for each $w\in W^{1,p}(\rn)$, set $w_n:=(w)^{-t_n}$. Then one can deduce that
\begin{equation*}
\nabla (J_V-J_{\infty})(u_n)[w]=\int_{\mathbb{R}^3}V(\frac{x}{t_n})|v_n|^{p-2}v_n w_ndx,
\end{equation*}
which yields that
\begin{equation*}
\nabla J_V(u_n)[w]=\nabla\tilde{J}_V(v_n,t_n)[(w_n,1)]+o(1)||w_n||.
\end{equation*}
Moreover, $\int_{\mathbb{R}^N}v_nwdx=0$ is equivalent to $\int_{\mathbb{R}^N}u_nw_ndx=0$. By the definition of $w_n$, we know that for $n$ large, $||w_n||^p\leq 2||w||^p$, thus we have $\nabla J_V(u_n)[w]=o_n(1)\|w\|$. 
In particular, differentiating in $t$ we have
\begin{equation*}
\partial_t|_{t=1}\big(\int_{\rn}V(x)t^{Np/2}|u|^p(tx)dx\big)=\int_{\rn}V(x)(\frac{Np}{2}|u|^p+p |u|^{p-2}u\nabla u\cdot x)dx\qquad\forall\, u\in W^{1,p}(\rn)
\end{equation*}
and
\begin{equation*}
\partial_t|_{t=1} I(u^t)=\frac{p(N+2)-2N}{2p}\|\nabla u\|_p^p+\frac{N(p-2)}{2p}\|u\|^p_p-\frac{N(q-2)}{2q}\|u\|_q^q\qquad\forall\, u\in W^{1,p}(\rn),
\end{equation*}
hence we can see that $P(u_n)\to 0\text{ as}~n\rightarrow\infty$,
which means $u_n$ almost satisfies the Pohozaev identity. 
This completes the proof.
\end{proof}
\begin{lemma}\label{lem3.1}
Assume that (\ref{est-Vle0}) holds. Let $\{u_n\}\subset \mathcal{S}_{\rho}$ be the $(PS)$ sequence for $J_V$ obtained in Lemma \ref{lemma-PS-seq}.  Then $\{u_n\}$ is bounded in $W^{1,p}(\mathbb{R}^N)$. Moreover, there exists $\rho^* >0$ and a sequence of Lagrange multipliers
\begin{equation*}
\lambda_n:=-\frac{\nabla J_V(u_n)[u_n]}{\rho^2}
\end{equation*}
such that $\lambda_n\rightarrow \lambda>0$ for any $\rho\in(0,\rho^*)$.
\end{lemma}
\begin{proof}
Let $u_n$ be the PS-sequence  constructed in Lemma \ref{lemma-PS-seq}. Set
\begin{equation}
\begin{aligned}
&A_n:=\|\nabla u_n\|_p^p,\,B_n:=\|u_n\|^p_p,\,C_n:=\int_{\rn}V(x)|u_n|^pdx,\\
&D_n:=\int_{\rn}V(x)|u_n|^{p-2}u_n\nabla u_n\cdotp x dx,\,E_n:=\|u_n\|_q^q.    
\end{aligned}
\end{equation}
Then relations \eqref{syst-PS-seq} are equivalent to 
\begin{equation}
\label{syst-PS-seq-explicit}
\begin{aligned}
&A_n+B_n+C_n-\frac{p}{q}E_n=p m_{V,\rho}+o(1) \\
&\frac{p(N+2)-2N}{p}A_n+\frac{N(p-2)}{p}B_n-\frac{N(q-2)}{q}E_n+\frac{Np}{2}C_n+pD_n=o(1)\\
&A_n+B_n+C_n+\lambda_n\rho^2-E_n=o(1)(A_n^{1/p}+B_n^{1/p}).
\end{aligned}    
\end{equation}
Subtracting the second equation to the first one in \eqref{syst-PS-seq-explicit} we can see that
$$\frac{2N-p(N+1)}{p}A_n+\frac{p-N(p-2)}{p}B_n+\frac{2-Np}{2}C_n-pD_n+\frac{N(q-2)-p}{q}E_n=pm_{V,\rho}+o(1).$$
Using once again the first relation in \eqref{syst-PS-seq-explicit}, we have
\begin{equation}
\frac{Nq-p(N+2)}{p}A_n+\frac{N(q-p)}{p}B_n+N\frac{2(q-2)-p^2}{2p}C_n-pD_n-N(q-2)m_{V,\rho}=o(1)   
\end{equation}
Note that $Nq-p(N+2)>0$ since $q>p\frac{N+2}{N}$.
Using \eqref{est-CD}, the terms involving the potential are estimated by
$$pD_n-N\frac{2(q-2)-p^2}{2p}C_n\le \left(pS_p^{-\frac{p-1}{p}}\|\tilde{W}\|_{\frac{N}{p-1}}+N\frac{(2(q-2)-p^2)^+}{2p}S_p^{-1}\|V\|_{\frac{N}{p}}\right)A_n,$$
which yields that there exists $n_0>0$ such that
\begin{equation}\notag
\begin{aligned}
&\left(\frac{Nq-p(N+2)}{p}-N\frac{(2(q-2)-p^2)^+}{2p}S_p^{-1}\|V\|_{\frac{N}{p}}-pS_p^{-\frac{p-1}{p}}\|\tilde{W}\|_{\frac{N}{p-1}}\right)A_n\\
&\le -\frac{N(q-p)}{p}B_n+N(q-2)m_{V,\rho}+o(1)\le 2N(q-2)c_\rho\qquad\forall\,n\ge n_0,
\end{aligned}
\end{equation}
since $B_n\ge 0$. Thanks to \eqref{est-Vle0}, which implies that
$$N\frac{(2(q-2)-p^2)^+}{2p}S_p^{-1}\|V\|_{\frac{N}{p}}+pS_p^{-\frac{p-1}{p}}\|\tilde{W}\|_{\frac{N}{p-1}}<\frac{Nq-p(N+2)}{p},$$
this yields that $A_n$ is bounded, since $\frac{2(q-2)-p^2}{2}\le q-p-2$ for $p\ge 2$. Due to the Sobolev embeddings, $E_n$ is also bounded. Hence, due to \eqref{est-CD}, $C_n$ and $D_n$ are bounded as well. Finally, the Pohozaev identity (namely the second relation in \eqref{syst-PS-seq-explicit}) yields that $B_n$ is also bounded, so that $\lambda_n$ is bounded too.

Since we see that $u_n$ is bounded in $W^{1,p}(\rn)$ and in $L^2(\rn)$, then there exists $u\in W^{1,p}(\rn)$ and $\lambda$ such that up to a subsequence (relabelled by the same indices) $u_n\to u$ weakly in $W^{1,p}(\rn)$ and in $L^2(\rn)$ and $\lambda_n\to \lambda$. Due to the weak convergence, the pair $(u,\lambda)$ satisfies the equation
\begin{equation}
\label{eq-u}
-\Delta u+(1+V(x))|u|^{p-2}u+\lambda u=|u|^{q-2}u\qquad\text{in}\,\rn.    
\end{equation}
Let us show that $\lambda>0$. 
Testing equation \eqref{eq-u} with $u$ itself and using Lemma \ref{lemma-Poho}, it is possible to see that the relations
\begin{equation}
\label{syst-u}
\begin{aligned}
&A+B+C+\lambda\rho^2=E\\
&\frac{p(N+2)-2N}{p}A+\frac{N(p-2)}{p}B-\frac{N(q-2)}{q}E+\frac{Np}{2}C+pD=0,
\end{aligned}
\end{equation}
are fulfilled, where we have set
$$A:=\|\nabla u\|^p_p,\, B:=\|u\|^p_p,\,C:=\int_{\rn}V(x)|u|^pdx,\,D:=\int_{\rn}V(x)u\nabla u\cdotp x dx,\, E:=\|u\|^q_q.$$
The Pohozaev identity gives
\begin{equation}\notag
\begin{aligned}
\frac{p(N+2)-2N}{p}A&\le\frac{p(N+2)-2N}{p}A+\frac{N(p-2)}{p}B\\
&=\frac{N(q-2)}{q}E-\frac{Np}{2}C-pD.
\end{aligned}    
\end{equation}
Using the Gagliardo-Nirenberg inequality we have
$$E\le \kappa\|\nabla u\|_p^{\theta q}\|u\|_2^{(1-\theta)q}=\kappa A^{\theta\frac{q}{p}}\rho^{(1-\theta)q}.$$
Moreover, \eqref{est-CD} gives
$$|C|\le S_p^{-1}\|V\|_{N/p}A,\qquad |D|\le S_p^{-\frac{p-1}{p}}\|\tilde{W}\|_{\frac{N}{p-1}}A.$$
As a consequence
$$\left(\frac{p(N+2)-2N}{p}-\frac{Np}{2}S_p^{-1}\|V\|_{N/p}-pS_p^{-\frac{p-1}{p}}\|\tilde{W}\|_{\frac{N}{p-1}}\right)A^{1-\theta\frac{q}{p}}\le\tilde{\kappa}\rho^{(1-\theta)q}.$$
Using that $1-\theta\frac{q}{p}<0$ and
\begin{equation}
\label{extra-cond-V1}
\frac{p(N+2)-2N}{p}-\frac{Np}{2}S_p^{-1}\|V\|_{N/p}-pS_p^{-\frac{p-1}{p}}\|\tilde{W}\|_{\frac{N}{p-1}}>0,
\end{equation}
due to \eqref{est-Vle0} and the fact that
$$2\left(\frac{N}{q}-\frac{N-p}{p}\right)<\frac{p(N+2)-2N}{p},$$
we deduce that $A\to\infty$ as $\rho\to 0$.\\

Multiplying the first relation in \eqref{syst-u} by $\frac{N(q-2)}{2q}$, the second one by $1/2$ and subtracting, it is possible to see that
\begin{equation}
\begin{aligned}
\lambda\rho^2\left(\frac{N}{2}-\frac{N}{q}\right)&=\left(\frac{N}{q}-\frac{N-p}{p}\right)A-\left(\frac{N}{p}-\frac{N}{q}\right)B+\left(\frac{N}{q}+\frac{N}{4}(p-2)\right)C+\frac{p}{2}D \\
&\ge\left(\frac{N}{q}-\frac{N-p}{p}-\left(\frac{N}{q}+\frac{N}{4}(p-2)\right)S_p^{-1}\|V\|_{N/p}-\frac{p}{2}S_p^{-\frac{p-1}{p}}\|\tilde{W}\|_{\frac{N}{p-1}}\right)A\\
&-\left(\frac{N}{p}-\frac{N}{q}\right)B .
\end{aligned}
\end{equation}
In order to control the negative term involving $B$, we use the interpolation inequality $$\|u\|^p_p\le\|u\|_2^{\beta p}\|u\|_q^{(1-\beta)q},\qquad\frac{1}{p}=\frac{\beta}{2}+\frac{1-\beta}{q},\,\beta\in(0,1),$$
and the Gagliardo-Nirenberg inequality we can see that 
\begin{equation}\notag
\begin{aligned}
&B\le \|u\|_2^{\beta p}\|u\|_q^{(1-\beta)q}\le \kappa \|u\|_2^{\beta p}(\|u\|_2^{1-\theta}\|\nabla u\|_p^\theta)^{(1-\beta)p}\\
&=\kappa\rho^{\beta p+(1-\theta)(1-\beta)p}A^{\theta (1-\beta)}\le \kappa\rho^{\beta p+(1-\theta)(1-\beta)p}A,   
\end{aligned}
\end{equation}
for $\rho>0$ sufficiently small, since $(1-\beta)\theta\in(0,1)$ and $A\to\infty$ as $\rho\to 0$. As a consequence, $\lambda>0$ if 
\begin{equation}
\label{cond-V-lambda>0}
\frac{N}{q}-\frac{N-p}{p}-\left(\frac{N}{q}+\frac{N}{4}(p-2)\right)S_p^{-1}\|V\|_{N/p}-\frac{p}{2}S_p^{-\frac{p-1}{p}}\|\tilde{W}\|_{\frac{N}{p-1}}>0
\end{equation}
and $\rho>0$ is small enough. It is possible to see that \eqref{cond-V-lambda>0} is true thanks to \eqref{est-Vle0}, since $$2\left(\frac{1}{q}+\frac{1}{4}(p-2)\right)<\frac{p}{2}.$$ 
\end{proof}

Now we can conclude the proof of Theorem \ref{th-Vle0} by means of the splitting Lemma, which gives compactness of the Palais-Smale sequence found in Lemma \ref{lemma-PS-seq}.
\begin{proof}
First we note that, thanks to the fact that the Palais-Smale sequence $u_n$ constructed in Lemma \ref{lem3.1} is bounded in $W^{1,p}(\rn)\cap L^2(\rn)$ and $\lambda_n\to\lambda$, there exists a subsequence, still denoted by $u_n$, and a solution $u\in W^{1,p}(\rn)\cap L^2(\rn)$ to the equation
$$-\Delta_p u+(1+V(x))|u|^{p-2}u+\lambda u=|u|^{q-2}u\qquad\text{in}\,\rn$$
such that $u_n\rightharpoonup u$ in $W^{1,p}(\rn)\cap L^2(\rn)$.\\

Thanks to the fact that $\lambda>0$ and the assumptions about $V$, we can apply the splitting Lemma \eqref{splitting-lemma} which yields that 
\begin{equation}\notag
u_n=u+\sum_{j=1}^kw^j(\cdot-y_n^j )+o(1)~\text{strongly~in}~W^{1,p}(\mathbb{R}^N).
\end{equation}
where $w^j$ are solutions to the limit equation $$-\Delta_p w+|w|^{p-2}w+\lambda w=|w|^{q-2}w\qquad\text{in}\,\rn,$$
\begin{equation}\notag
\|u_n\|_2^2=\|u\|_2^2+\sum_{j=1}^k\|w^j\|_2^2+o(1)
\end{equation}
and
\begin{equation}\notag
J_{V,\lambda}(u_n)= J_{V,\lambda}(u)+\sum_{j=1}^kI_{\lambda}(w^j)+o(1) .
\end{equation} 
As a consequence we have
$$J_V(u_n)=J_V(u)+\sum_{j=1}^k I(w^j)+o(1).$$
If we assume by contradiction that $k\ge 1$, letting $n\to\infty$ we have
$$m_{V,\rho}=J_V(u)+\sum_{j=0}^k I(w^j).$$
Using that $u$ is a solution to \eqref{main-eq-not-normalised}, Lemma \ref{lem3.1} yields that $J_V(u)>0$. Moreover, using Remark \ref{rem-least-energy-sol} and the fact that $c_\rho$ is nonincreasing in $\rho$, we have $I(w^j)\ge c_{\|w^j\|_2}\ge c_\rho$. Finally, using the fact that $V\ne 0$, Lemma \ref{lemma-mpg} gives
$$c_\rho>m_{V,\rho}=J_V(u)+\sum_{j=0}^k I(w^j)>\sum_{j=0}^k c_{\|w^j\|_2}\ge c_\rho,$$
a contradiction. This yields that $k=0$, or equivalently $u_u\to u$ in $W^{1,p}(\rn)\cap L^2(\rn)$, so that $\|u\|_2=\rho$. This concludes the proof of Theorem \ref{th-Vle0}.
\end{proof}

\subsection{The radial case: the proof of Lemma \ref{lemma-compact-rad}}\label{subsec-rad}

The proof parallels the one of Lemma $1.1$ and Theorem $1.2$ of \cite{kavian1993introduction}.
\begin{proof}
We start by proving that there exists a constant $C(N,p)>0$ such that
\begin{equation}\label{point-dec-rad}
|\varphi(x)|\le C(N,p)|x|^{-\frac{N-1}{p}}\|\varphi\|\qquad\forall \,\varphi\in W^{1,p}(\rn).
\end{equation}
Since the space of radial $C^\infty_c(\rn)$-functions is dense in $W^{1,p}_{rad}(\rn)$, it is enough to prove that \eqref{point-dec-rad} holds true for any radial function $\varphi\in C^\infty_c(\rn)$. For this purpose, writing $\varphi(x)=\psi(r)$, where $r:=|x|$, we note that
\begin{equation}\notag
\begin{aligned}
|\varphi(x)|^p&=|\psi(r)|^p=-p\int_r^\infty |\psi(s)|^{p-2}\psi(s)|\psi'(s)ds\le\int_r^\infty |\psi'(s)|^p ds+\frac{p'}{p}\int_r^\infty |\psi'(s)|^p ds\\
&\le c_p\int_r^\infty s^{-(N-1)}(|\psi'(s)|^p+|\psi(s)|^p)s^{N-1}ds\le c_p r^{1-N}\|\varphi\|^p,
\end{aligned}
\end{equation}
which proves the decay estimate \eqref{point-dec-rad}.\\

Now let us consider a bounded sequence $\{u_n\}\subset W^{1,p}(\rn)$. Then, up to a subsequence, $u_n\rightharpoonup u$ weakly in $W^{1,p}(\rn)$, for some $u\in W^{1,p}(\rn)$. Then, for any $\varepsilon>0$, there exists $R_0(\varepsilon)>0$ such that, for any $R\ge R_0(\varepsilon)$ we have
\begin{equation}\notag
\begin{aligned}
&\int_{\rn\setminus B_R(0)}|u_n-u|^qdx\le\sup_{\rn\setminus B_R(0)}(|u_n-u|^{q-p})\left(\int_{\rn\setminus B_R(0)}|u_n-u|^pdx\right)\\
&\le C(p,N)R^{-\frac{(q-p)(N-1)}{p}}(\sup_{k}\|u_k\|_p^p+\|u\|_p^p)\le \tilde{C}(p,N)R^{-\frac{(q-p)(N-1)}{p}}<\varepsilon,\,\forall\, n\in \mathbb{N},
\end{aligned}
\end{equation}
since $u_n$ is bounded in $W^{1,p}(\rn)$. Moreover, since $u_n\to u$ strongly in $L^q(B_R(0))$, there exists $n_0(\varepsilon)>0$ such that 
$$\int_{B_R(0)}|u_n-u|^q dx\le \varepsilon\qquad\forall\,n\ge n_0(\varepsilon),$$
which concludes the proof.
\end{proof}



\section*{Data availability}
Data sharing not applicable to this article as no datasets were generated or analysed during the current study.

\section*{Declarations}
{\bf Competing interest statement.} The authors declare that they have no competing interests regarding this manuscript.

\bmhead{Acknowledgements. 
The authors are  grateful for the reviews by anonymous referees to improve the quality of the article}

\begin{appendices}
\section{}
In this section we compute the decay rate of the solution of the following problem
\begin{align}
    -\Delta_p u +u^{r-1} &=u^{q-1}\quad\text{in}\ \rn,\\
    u&>0,
\end{align}
for $1<p<q<r<p^*$, which gives \eqref{point-dec-rad}. Using and $u$ is radial in the form $u(x)=v(|x|)$, we can see that $v$ solves
\begin{align}\label{eq:rad:V0}
    -\frac{1}{s^{N-1}}(|v'|^{p-2}v' s^{N-1})'=v^{q-1}-v^{r-1\qquad}\forall s\in[\rho,\infty).
\end{align}
Assuming that $v'<0$ on $(\rho_0,\infty)$ for some $\rho_0>\rho$, we have from~\eqref{eq:rad:V0}
\begin{align}\label{eq:rad:v1}
    \left(s^{N-1}(-v')^{p-1}\right)'= s^{N-1}(v^{q-1}-v^{r-1}).
\end{align}
After integrating on $(\rho_0,s)$, we obtain
\begin{align}
s^{N-1}(-v'(s))^{p-1}\ge s^{N-1}(-v'(s))^{p-1}-\rho_0^{N-1}(-v'(\rho_0))^{p-1}
\stackrel{\eqref{eq:rad:v1}}{=}
\int_{\rho_0}^{s}t^{N-1}(v^{q-1}(t)-v^{r-1}(t))\, dt. \nonumber
\end{align}
for any $s\ge\rho_0$. Since $v(t)\in(0,\delta)$ for $t>\rho_0$, using that $r>q$, we have 
\begin{align}
&s^{N-1}(-v'(s))^{p-1}\ge c\int_{\rho_0}^s t^{N-1}v^{q-1}(t)dt
\ge cv^{q-1}(s)\int_{\rho_0}^s t^{N-1}\, dt\\
&= cv^{q-1}(s)(s^N-\rho_0^n)/N\ge cv^{q-1}(s)s^N\qquad\forall\, s>\rho_1 \nonumber\\
\end{align}
for some $\rho_1>\rho_0$ large enough. Integrating both side form $\rho_1$ to $s$ and using that $p<q$, we obtain
\begin{equation}
\begin{aligned}
v^{1-\frac{q-1}{p-1}}(s)&=v^{\frac{p-q}{p-1}}(s)\ge v^{\frac{p-q}{p-1}}(s)-v^{\frac{p-q}{p-1}}(\rho_1)=\int_{\rho_1}^s \frac{d}{dt}(v^{\frac{p-q}{p-1}}(t))dt\\
&=-c\int_{\rho_1}^s v^{-\frac{q-1}{p-1}}(t)v'(t)dt\ge\kappa\int_{\rho_1}^s t^{\frac{1}{p-1}}dt=\kappa(p-1)(s^{\frac{p}{p-1}}-\rho_1^{\frac{p}{p-1}})>0,
\end{aligned}
\end{equation}
which is equivalent to
$$v(s)\le \tilde{\kappa}s^{-\frac{p}{q-p}}$$
for $s$ large enough.\\

\end{appendices}
\bibliography{ref}


\begin{thebibliography}{13}
\ifx \bisbn   \undefined \def \bisbn  #1{ISBN #1}\fi
\ifx \binits  \undefined \def \binits#1{#1}\fi
\ifx \bauthor  \undefined \def \bauthor#1{#1}\fi
\ifx \batitle  \undefined \def \batitle#1{#1}\fi
\ifx \bjtitle  \undefined \def \bjtitle#1{#1}\fi
\ifx \bvolume  \undefined \def \bvolume#1{\textbf{#1}}\fi
\ifx \byear  \undefined \def \byear#1{#1}\fi
\ifx \bissue  \undefined \def \bissue#1{#1}\fi
\ifx \bfpage  \undefined \def \bfpage#1{#1}\fi
\ifx \blpage  \undefined \def \blpage #1{#1}\fi
\ifx \burl  \undefined \def \burl#1{\textsf{#1}}\fi
\ifx \doiurl  \undefined \def \doiurl#1{\url{https://doi.org/#1}}\fi
\ifx \betal  \undefined \def \betal{\textit{et al.}}\fi
\ifx \binstitute  \undefined \def \binstitute#1{#1}\fi
\ifx \binstitutionaled  \undefined \def \binstitutionaled#1{#1}\fi
\ifx \bctitle  \undefined \def \bctitle#1{#1}\fi
\ifx \beditor  \undefined \def \beditor#1{#1}\fi
\ifx \bpublisher  \undefined \def \bpublisher#1{#1}\fi
\ifx \bbtitle  \undefined \def \bbtitle#1{#1}\fi
\ifx \bedition  \undefined \def \bedition#1{#1}\fi
\ifx \bseriesno  \undefined \def \bseriesno#1{#1}\fi
\ifx \blocation  \undefined \def \blocation#1{#1}\fi
\ifx \bsertitle  \undefined \def \bsertitle#1{#1}\fi
\ifx \bsnm \undefined \def \bsnm#1{#1}\fi
\ifx \bsuffix \undefined \def \bsuffix#1{#1}\fi
\ifx \bparticle \undefined \def \bparticle#1{#1}\fi
\ifx \barticle \undefined \def \barticle#1{#1}\fi
\bibcommenthead
\ifx \bconfdate \undefined \def \bconfdate #1{#1}\fi
\ifx \botherref \undefined \def \botherref #1{#1}\fi
\ifx \url \undefined \def \url#1{\textsf{#1}}\fi
\ifx \bchapter \undefined \def \bchapter#1{#1}\fi
\ifx \bbook \undefined \def \bbook#1{#1}\fi
\ifx \bcomment \undefined \def \bcomment#1{#1}\fi
\ifx \oauthor \undefined \def \oauthor#1{#1}\fi
\ifx \citeauthoryear \undefined \def \citeauthoryear#1{#1}\fi
\ifx \endbibitem  \undefined \def \endbibitem {}\fi
\ifx \bconflocation  \undefined \def \bconflocation#1{#1}\fi
\ifx \arxivurl  \undefined \def \arxivurl#1{\textsf{#1}}\fi
\csname PreBibitemsHook\endcsname

\bibitem[\protect\citeauthoryear{Wang et~al.}{2021}]{WLZL}
\begin{barticle}
\bauthor{\bsnm{Wang}, \binits{W.}},
\bauthor{\bsnm{Li}, \binits{Q.}},
\bauthor{\bsnm{Zhou}, \binits{J.}},
\bauthor{\bsnm{Li}, \binits{Y.}}:
\batitle{Normalized solutions for $p$-laplacian equations with a $l^2$-supercritical growth}.
\bjtitle{Annals of Functional Analysis}
\bvolume{12},
\bfpage{1}--\blpage{19}
(\byear{2021})
\end{barticle}
\endbibitem

\bibitem[\protect\citeauthoryear{Cingolani and Jeanjean}{2019}]{cingolani2019stationary}
\begin{barticle}
\bauthor{\bsnm{Cingolani}, \binits{S.}},
\bauthor{\bsnm{Jeanjean}, \binits{L.}}:
\batitle{Stationary waves with prescribed $l^2$-norm for the planar schr{\"o}dinger--poisson system}.
\bjtitle{SIAM Journal on Mathematical Analysis}
\bvolume{51}(\bissue{4}),
\bfpage{3533}--\blpage{3568}
(\byear{2019})
\end{barticle}
\endbibitem

\bibitem[\protect\citeauthoryear{Peng and Rizzi}{2024}]{PR}
\begin{botherref}
\oauthor{\bsnm{Peng}, \binits{X.}},
\oauthor{\bsnm{Rizzi}, \binits{M.}}:
Normalized solutions of mass supercritical schr\"{o}dinger-poisson equation with potential.
submitted to "Calculus of variations and PDEs",
243--266
(2024)
\end{botherref}
\endbibitem

\bibitem[\protect\citeauthoryear{Bartsch et~al.}{2021}]{bartsch2021normalized}
\begin{barticle}
\bauthor{\bsnm{Bartsch}, \binits{T.}},
\bauthor{\bsnm{Molle}, \binits{R.}},
\bauthor{\bsnm{Rizzi}, \binits{M.}},
\bauthor{\bsnm{Verzini}, \binits{G.}}:
\batitle{Normalized solutions of mass supercritical schr{\"o}dinger equations with potential}.
\bjtitle{Communications in Partial Differential Equations}
\bvolume{46}(\bissue{9}),
\bfpage{1729}--\blpage{1756}
(\byear{2021})
\end{barticle}
\endbibitem

\bibitem[\protect\citeauthoryear{Rizzi}{2017}]{R2}
\begin{barticle}
\bauthor{\bsnm{Rizzi}, \binits{M.}}:
\batitle{Clifford tori and the singularly perturbed cahn--hilliard equation}.
\bjtitle{Journal of Differential Equations}
\bvolume{262}(\bissue{10}),
\bfpage{5306}--\blpage{5362}
(\byear{2017})
\end{barticle}
\endbibitem

\bibitem[\protect\citeauthoryear{Malchiodi et~al.}{2018}]{malchiodi2018periodic}
\begin{barticle}
\bauthor{\bsnm{Malchiodi}, \binits{A.}},
\bauthor{\bsnm{Mandel}, \binits{R.}},
\bauthor{\bsnm{Rizzi}, \binits{M.}}:
\batitle{Periodic solutions to a cahn--hilliard--willmore equation in the plane}.
\bjtitle{Archive for Rational Mechanics and Analysis}
\bvolume{228},
\bfpage{821}--\blpage{866}
(\byear{2018})
\end{barticle}
\endbibitem

\bibitem[\protect\citeauthoryear{Agudelo et~al.}{2015}]{agudelo2015solutions}
\begin{barticle}
\bauthor{\bsnm{Agudelo}, \binits{O.}},
\bauthor{\bsnm{Pino}, \binits{M.}},
\bauthor{\bsnm{Wei}, \binits{J.}}:
\batitle{Solutions with multiple catenoidal ends to the allen--cahn equation in $r^3$}.
\bjtitle{Journal de Math{\'e}matiques Pures et Appliqu{\'e}es}
\bvolume{103}(\bissue{1}),
\bfpage{142}--\blpage{218}
(\byear{2015})
\end{barticle}
\endbibitem

\bibitem[\protect\citeauthoryear{Agudelo et~al.}{2022}]{agudelo2022doubling}
\begin{barticle}
\bauthor{\bsnm{Agudelo}, \binits{O.}},
\bauthor{\bsnm{Kowalczyk}, \binits{M.}},
\bauthor{\bsnm{Rizzi}, \binits{M.}}:
\batitle{Doubling construction for $o(m)\times o(n)$ invariant solutions to the allen--cahn equation}.
\bjtitle{Nonlinear Analysis}
\bvolume{216},
\bfpage{112705}
(\byear{2022})
\end{barticle}
\endbibitem

\bibitem[\protect\citeauthoryear{Agudelo and Rizzi}{2022}]{agudelo2022k}
\begin{barticle}
\bauthor{\bsnm{Agudelo}, \binits{O.}},
\bauthor{\bsnm{Rizzi}, \binits{M.}}:
\batitle{k-ended $o(m)\times o(n)$ invariant solutions to the allen-cahn equation with infinite morse index}.
\bjtitle{Journal of Functional Analysis}
\bvolume{283}(\bissue{5}),
\bfpage{109561}
(\byear{2022})
\end{barticle}
\endbibitem

\bibitem[\protect\citeauthoryear{Kavian}{1993}]{kavian1993introduction}
\begin{botherref}
\oauthor{\bsnm{Kavian}, \binits{O.}}:
Introduction {\`a} la th{\'e}orie des points critiques: et applications aux probl{\`e}mes elliptiques.
(No Title)
(1993)
\end{botherref}
\endbibitem

\bibitem[\protect\citeauthoryear{Serrin and Tang}{2000}]{serrin2000uniqueness}
\begin{botherref}
\oauthor{\bsnm{Serrin}, \binits{J.}},
\oauthor{\bsnm{Tang}, \binits{M.}}:
Uniqueness of ground states for quasilinear elliptic equations.
Indiana University Mathematics Journal,
897--923
(2000)
\end{botherref}
\endbibitem

\bibitem[\protect\citeauthoryear{Rizzi and Smyrnelis}{2024}]{rizzi2024some}
\begin{barticle}
\bauthor{\bsnm{Rizzi}, \binits{M.}},
\bauthor{\bsnm{Smyrnelis}, \binits{P.}}:
\batitle{Some rigidity results and asymptotic properties for solutions to semilinear elliptic pde}.
\bjtitle{Nonlinear Analysis}
\bvolume{247},
\bfpage{113610}
(\byear{2024})
\end{barticle}
\endbibitem

\bibitem[\protect\citeauthoryear{Ghoussoub}{1993}]{G1993}
\begin{botherref}
\oauthor{\bsnm{Ghoussoub}, \binits{N.}}:
Duality and perturbation methods in critical point theory.with appendices by david robinson.
Cambridge tracts in mathematics, 107. Cambridge: Cambridge University Press
(1993)
\end{botherref}
\endbibitem

\end{thebibliography}
\end{document}